\documentclass[14pt,reqno]{amsart}
\usepackage[margin=3cm]{geometry}
\usepackage{setspace}
\usepackage{mathrsfs}
\usepackage{amssymb}
\usepackage{amsmath}
\usepackage{color}
\theoremstyle{plain}
\newtheorem{theorem}{Theorem}[section]

\newtheorem{lemma}[theorem]{Lemma}

\newtheorem{corollary}[theorem]{Corollary}

\theoremstyle{definition}
\newtheorem{definition}[theorem]{Definition}

\theoremstyle{remark}
\newtheorem{remark}[theorem]{Remark}

\begin{document}

\title[Well-posedness and blow-up criteria]
{On the Cauchy problem and asymptotic behavior for a three-component Novikov system}

\author{Zhi-gang Li}
\address{College of sciences, China University of Mining and Technology, Beijing 100083, P. R. China}
\email{lzgcumtb@163.com}

\author{Zhonglong Zhao*}
\address{Department of Mathematics, North University of China, Taiyuan, Shanxi 030051, P. R. China}
\email{zhaozl@nuc.edu.cn}

\keywords{well-posedness, three-component Novikov system, asymptotic behavior, traveling wave solution}

\begin{abstract}
This paper is mainly concerned with the well-posedness and exponential decay of solution for a integrable three-component Novikov system, which admits bi-Hamiltonian structure and infinitely many conserved quantities. The local well-posedness of this system in critical Besov space is established. The exponential decay of the solutions at infinity is also proved.\\
Correspondence should be addressed to Zhonglong Zhao; zhaozl@nuc.edu.cn
\end{abstract}

\maketitle

\section{Introduction}
In this paper, we consider the following three-component Novikov system (3NS) :
\begin{equation}
\begin{cases}
\rho_t+(\rho uv)_x=0,\\
m_t+3mu_xv+m_xuv+\rho^2u=0,\\
n_t+3nuv_x+n_xuv-\rho^2v=0,\\
m=u-u_{xx},n=v-v_{xx},
\end{cases}
\end{equation}
which was constructed by Li[1]. The author showed that it is equivalent to the zero-curvature equation
$$U_t-V_x+[U,V]=0,$$
where the spacial and temporal 3$\times$3 matrices $U$ and $V$ are
\begin{equation*}
U=\left(\begin{matrix}
0 & 1 & 0\\
1+\lambda\rho^2 & 0 &m\\
\lambda n & 0 & 0
\end{matrix}\right),
V=\left(\begin{matrix}
\frac{1}{3\lambda}+uv_x & -uv & \frac{u}{\lambda}\\
u_xv_x-\lambda\rho^2 uv & \frac{1}{3\lambda}-u_xv & \frac{u_x}{\lambda}-muv\\
-\lambda nuv-v_x & v & u_xv-uv_x-\frac{2}{3\lambda}
\end{matrix}\right),
\end{equation*}
the real number $\lambda \in \mathbb{R}$ is spectral parameter.

System (1) is obvious an extension of the following two-component Novikov system (2NS) within $\rho=0$,
\begin{equation*}
\begin{cases}
m_t+3u_xvm+uvm_x=0,\\
n_t\ +3v_xun\ +uvn_x\ =0,\\
m=u-u_{xx}, n=v-v_{xx},
\end{cases}
\end{equation*}
which was shown by Geng and Xue [2], who proved the integrability by arising the zero curvature equation. The authors also supplied many significant results, such as Hamiltonian structure, infinite many conserved quantities and explicit multi-peakon traveling wave solutions. It should be noted that the bi-Hamiltonian structure was found by Li and Liu[3], which means 2NS is also integrable in Liouville sense. Himonas and Mantzavinos considered the well-posedness and uniformly continuous in[4], the global existence can refer to Li, Hu and Wu in [5].

As 2NS is a multi-component system, we can exploit some reductions to reduce it into some single-component equations. For $u=v$, $2N$ system is reduced to Novikov equation (NE) [6],
\begin{equation*}
m_t+u^2m_x+3uu_xm=0,\ \ \ m=u-u_{xx},
\end{equation*}
which was firstly constructed by Novikov via the symmetry classification method. The integrability of NE was shown by Hone and Wang, who proved that it is integrable both in Lax and Liouville sense, which means it admits bi-Hamiltonian structure, infinitely many conserved quantities and can also be derived by Lax equation [7]. They also showed that NE is associated to the negative flow in the Sawada-Kotera hierarchy. The Cauchy problem and ill-posedness of such equation can refer to[8-9], while for periodic case and for $s > \frac{5}{2}$, well-posedness had been proved by Tiglay[10], and the blow-up phenomena has studied in[11].

Another important reduction of 2NS is Degasperis-Procesi equation (DPE) [12] if we take $v=1$,
$$m_t+um_x+3u_xm=0, \ \ \ m=u-u_{xx}.$$
It was proposed by Degasperis and Procesi, who considered the asymptotic integrability to the following dispersive PDE,
\begin{equation*}
u_t- \alpha^2u_{xxt} + \gamma u_{xxx} +c_0u_x =(c_1u^2+c_2u_x^2+c_3uu_{xx})_x.
\end{equation*}
In fact, if we change the coefficient of term $u_xm$ to 2, it is just the famous Camassa-Holm equation (CHE)[13-15]. DPE is another integrable equation of b-family for $b=3$. As the same as CHE, DPE also arises bi-Hamiltonian structure, infinite many conserved quantities and peaked solutions, and it is connected with a negative flow in the Kaup-Kupershmidt hierarchy throw reciprocal transformation [16]. The well-posedness and stability of DPE have been shown in[17-18].

It should be noted that the nonlinearities in DPE are quadratic. However, for NE and 2NS, the nonlinearities are cubic, thus it is more appropriate to call system (1) three-component ``Novikov" system, rather than three-component ``Degasperis-Procesi" system the author used in [1]. Very recently, we have just studied the Cauchy problem of 3NS in Besov space $B_{p,r}^s$ with $s>\max\{\frac{1}{2},\frac{1}{p}\}$ in [19]. Note that the following embedding holds
$$H^s=B_{2,2}^s\hookrightarrow B_{2,1}^\frac{1}{2}\hookrightarrow H^\frac{1}{2} \hookrightarrow B_{2,\infty}^\frac{1}{2} \hookrightarrow H^{s_1},\ \ \ \forall s>\frac{1}{2}>s_1,$$
it is interested to consider the local well-posedness in critical Besov space $B_{2,1}^\frac{1}{2}$, and we obtain such property by the linear transport equations theory. Next, we study the exponential decay of solution with initial data $(u_0,v_0)\sim\mathcal{O}(e^{-\delta|x|}), \delta \in (0,1)$ and $(\rho_0,m_0,m_0)\sim\mathcal{O}(e^{-\delta|x|}), \delta \in (0,\infty)$ as $x\rightarrow \pm\infty$. Finally, we obtain the asymmetric property of nontrivial traveling wave solutions, which is very different from most CH type systems as they admit peaked solutions with symmetry axis $x=ct$.

The structure of our paper is organized as follows. In Section 2, the local well-posedness of 3NS in critical Besov space is obtained. In section 3, the exponential decay of solution is studied, within some suitable decay condition for initial data. The asymmetric property of traveling solution is established in last section.

\textbf{Notation.} All the spaces of functions we consider in this paper are over $\mathbb{R}$, we omit $\mathbb{R}$ from our notation. We say $f(x)\sim\mathcal{O}(e^{\delta x})$ as $x\rightarrow \infty$ means $\displaystyle\lim_{x\rightarrow \infty} \dfrac{|f(x)|}{e^{\delta x}} \le L$ for some $L>0$, and $f(x)\sim o(e^{\delta x})$ as $x\rightarrow \infty$ if $\displaystyle\lim_{x\rightarrow \infty} \dfrac{|f(x)|}{e^{\delta x}}=0$.

\section{Local well-posedness in critical Besov space $B_{2,1}^\frac{1}{2}$}
In this section, we establish local well-posedness of 3NS in critical Besov space, the associated initial value problem is
\begin{equation}
\begin{cases}
\rho_t+(\rho uv)_x=0,\\
m_t+3mu_xv+m_xuv+\rho^2u=0,\\
n_t+3nuv_x+n_xuv-\rho^2v=0,\\
m=u-u_{xx},n=v-v_{xx},\\
(\rho(t),u(t),v(t))|_{t=0}=(\rho_0,u_0,v_0).
\end{cases}
\end{equation}
Let's recall the local well-posedness of 3NS in common Besov space first.
\begin{lemma}[19]
Let $1 \le p,r \le \infty$, $s >max\{\tfrac{1}{p},\tfrac{1}{2}\}$, and $(\rho_0,m_0,n_0) \in (B_{p,r}^s)^3$. Then there exists some $T>0$, such that system (1) has a unique solution in $E_{p,r}^s(T)$, and the map $(\rho_0, m_0, n_0) \rightarrow (\rho, m, n)$ is H{\"o}lder continuous from a neighborhood of $(\rho, m_0, n_0)$ in $(B_{p,r}^s)^3$ into $E_{p,r}^{s'}$, for every $s'<s$ when $r=+\infty$ and $s'=s$ when $r<+\infty$, where $E_{p,r}^s(T)$ is defined by
\begin{equation*}
\begin{cases}
E_{p,r}^s(T)\ =[C([0,T];B_{p,r}^s) \cap C^1([0,T];B_{p,r}^{s-1})]^3, if \ r < +\infty\\
E_{p,\infty}^s(T)=[L^{\infty}([0,T];B_{p,\infty}^s) \cap Lip([0,T];B_{p,\infty}^{s-1})]^3
\end{cases}
\end{equation*}
and $E_{p,r}^s=\bigcap_{T>0}E_{p,r}^s(T).$
\end{lemma}

Our main result in this section is stated as follows.
\begin{theorem}
Suppose $(\rho_0,m_0,n_0)\in B_{2,1}^{\frac{1}{2}}$. Then there exists some $T>0$, such that 3NS has a unique solution in $C([0,T);B_{2,1}^{\frac{1}{2}})\bigcap C^1([0,T);B_{2,1}^{-\frac{1}{2}})$ such that the map
$$(\rho_0,m_0,n_0) \rightarrow (\rho,m,n):B_{2,1}^\frac{1}{2} \rightarrow C([0,T);B_{2,1}^{\frac{1}{2}})\bigcap C^1([0,T);B_{2,1}^{-\frac{1}{2}})$$
is H\"{o}lder continuous.
\end{theorem}

The reason we come back to the Cauchy problem is that in the proof of Lemma 2.1, one of the crucial technique is 1-D Moser inequality [x]
$$\|fg\|_{B_{p,r}^{s_1}} \le C\|f\|_{B_{p,r}^{s_1}}\|g\|_{B_{p,r}^{s_2}},$$
for $s_1 \le \tfrac{1}{p}, s_2 > \tfrac{1}{p}(s_2 \ge \tfrac{1}{p}\ if\ r=1)$,and $s_1+s_2>0$, which is failed in critical sense as $s_1=-\frac{1}{2}$ and $s_2=\frac{1}{2}$.
Thus we need another way to overcome this difficulty. In order to prove Theorem 2.2, we need the following two lemmas.

\begin{lemma}$[20]$(A priori estimates in Besov spaces) Consider the following transport equation
\begin{equation*}
\begin{cases}
\partial_t f+v\cdot \nabla f =g,\\
f|_{t=0}=f_0.
\end{cases}
\end{equation*}
Let $1 \le p \le p_1 \le \infty$, $1 \le r \le \infty$, $s \ge -d\cdot min(\tfrac{1}{p_1},\tfrac{1}{p'})$. For the solution $ f \in L^{\infty}([0,T];B_{p,r}^s(\mathbb{R}^d))$ of (3) with velocity $v, \nabla v \in  L^1([0,T];B_{p,r}^s(\mathbb{R}^d)\cap L^{\infty}(\mathbb{R}^d))$, initial data $f_0 \in B_{p,r}^s(\mathbb{R}^d)$ and $g\in L^1([0,T];B_{p,r}^s(\mathbb{R}^d))$, we have
$$\|f(t)\|_{B_{p,r}^s} \le \|f_0\|_{B_{p,r}^s}+\int_0^t\left(\|g(t')\|_{{B_{p,r}^s}}+CV'_{p_1}(t')\|f(t')\|_{{B_{p,r}^s}}\right)dt',$$
$$\|f(t)\|_{L_t^{\infty}(B_{p,r}^s)} \le \left(\|f_0\|_{B_{p,r}^s}+\int_0^t\exp(-CV_{p_1}(t'))\|g(t')\|_{{B_{p,r}^s}}dt'\right)\exp(CV_{p_1}),$$
where
\begin{equation*}
V_{p_1}(t)=\begin{cases}
\displaystyle\int_0^t\|\nabla v\|_{B_{p_1,\infty}^{\frac{d}{p_1}} \cap L^{\infty}} &,\  if\ \ \  s<1+\tfrac{d}{p_1},\\
\displaystyle\int_0^t\|\nabla v\|_{B_{p_1,r}^{s-1}}&,\  if\ \ \  s>1+\tfrac{d}{p_1}\  or\  s=1+\tfrac{d}{p_1}, r=1,\\
\displaystyle\int_0^t\|\nabla v\|_{B_{p_1,1}^{\frac{d}{p_1}}}  &,\  if\ \ \  s=-d\cdot \min\{\frac{1}{p'},\frac{1}{p_1}\}, r=\infty.
\end{cases}
\end{equation*}
\end{lemma}

\begin{lemma}
$[20]$(Osgood lemma) Let $\rho$ be a measurable function from $[t_0,T]$ to $[0,a]$, $\gamma$ is a locally integrable function from $[0,T]$ to $\mathbb{R}^+$, and $\mu$ is a continuous and nondecreasing function from $[0,a]$ to $\mathbb{R}^+$. Assume for some nonnegative real nunber $c$, the function $\rho$ satisfies
$$\rho(t) \le c+\int_{t_0}^t\gamma(s)\mu(\rho(s))ds,\ \ \ for \ \ \ a.e. \ \ \ t\in[t_0,T].$$
If $c$ is positive, then we have for a.e. $t\in[t_0,T]$
$$-\mathcal{N}(\rho(t))+\mathcal{N}(c) \le \int_{t_0}^t\gamma(s)ds \ \ \ with \ \ \ \mathcal{N}(x)=\int_x^a\frac{d\tau}{\mu(\tau)}.$$
If $c=0$ and $\mu$ satisfies $\displaystyle\int_0^a\frac{d\tau}{\mu(\tau)} = +\infty$, then $\rho=0$ a.e.
\end{lemma}

Now we break the proof of Theorem 2.2 into following steps.

\begin{proof}
\textbf{Step 1.}
Firstly, we construct approximate smooth solutions for some linear equations by the classical Friedrichs regularization method. Starting from $(\rho^{(1)},m^{(1)},n^{(1)})=(S_1\rho_0,S_1m_0,S_1n_0)$, we define by induction sequences $\{\rho^{(n)},m^{(n)},n^{(n)}\}_{n \in \mathbb{N}}$ by solving the following linear transport equations:

\begin{equation}
\begin{cases}
\rho^{(n+1)}_t\ +u^{(n)}v^{(n)}\rho^{(n+1)}_x=-\rho^{(n)}(u^{(n)}_xv^{(n)}+u^{(n)}v^{(n)}_x),\\
m^{(n+1)}_t\ +u^{(n)}v^{(n)}m^{(n+1)}_x=-3m^{(n)}u^{(n)}_xv^{(n)}-(\rho^{(n)})^2u^{(n)},\\
n^{(n+1)}_t+u^{(n)}v^{(n)}n^{(n+1)}_x=-3n^{(n)}u^{(n)}v^{(n)}_x+(\rho^{(n)})^2v^{(n)},\\
(\rho^{(n+1)},m^{(n+1)},n^{(n+1)})|_{t=0}=S_{n+1}\rho_0,S_{n+1}m_0,S_{n+1}n_0,
\end{cases}
\end{equation}

where the operator $S_{n+1}f=\displaystyle\sum_{j\ge -1}^n \Delta_k f$, $\Delta_k$ is the dyadic operator.

Suppose that $(m^{(n)},n^{(n)}) \in L^{\infty}([0,T];B_{2,1}^\frac{1}{2})$. Since $(u^{(n)},v^{(n)})=(1-\partial^2)^{-1}(m^{(n)},n^{(n)})$, it follows that $(u^{(n)},v^{(n)}) \in L^{\infty}([0,T];B_{2,1}^\frac{5}{2})$. By the theory of transport equations, we obtain that $(m^{(n+1)},n^{(n+1)}) \in L^{\infty}([0,T];B_{2,1}^\frac{1}{2})$. For more details of transport theorem, one can refer to Chapter 3 in [1].

\textbf{Step 2.}
Next, we show that for some fixed positive number $T>0$, the sequences $\{\rho^{(n)},m^{(n)},n^{(n)}\}$ are uniformly bounded in $[C([0,T);B_{2,1}^{\frac{1}{2}})\bigcap C^1([0,T);B_{2,1}^{-\frac{1}{2}})]^3$, by virtue of Lemma 2.3, with $A_{n,s}=\displaystyle\int_0^t\|(u^{(n)}v^{(n)})_x\|_{B_{2,\infty}^\frac{1}{2}\bigcap L^\infty}d\tau$, we have
\begin{equation}
\begin{aligned}
\|\rho^{(n+1)}\|_{L_t^\infty(B_{2,1}^\frac{1}{2})} &\le e^{CA_{n,s}}\Big(\|S_{n+1}\eta_0\|_{B_{2,1}^\frac{1}{2}}+\int_0^te^{-CA_{n,s}(\tau)}\|-\rho^{(n)}(u^{(n)}_xv^{(n)}+u^{(n)}v^{(n)}_x)\|_{B_{2,1}^\frac{1}{2}}d\tau\Big),\\
\|m^{(n+1)}\|_{L_t^\infty(B_{2,1}^\frac{1}{2})} &\le e^{CA_{n,s}}\Big(\|S_{n+1}m_0\|_{B_{2,1}^\frac{1}{2}}+\int_0^te^{-CA_{n,s}(\tau)}\|-3m^{(n)}u^{(n)}_xv^{(n)}-(\rho^{(n)})^2u^{(n)}\|_{B_{2,1}^\frac{1}{2}}d\tau\Big),\\
\|n^{(n+1)}\|_{L_t^\infty(B_{2,1}^\frac{1}{2})} &\le e^{CA_{n,s}}\Big(\|S_{n+1}n_0\|_{B_{2,1}^\frac{1}{2}}+\int_0^te^{-CA_{n,s}(\tau)}\|-3n^{(n)}u^{(n)}v^{(n)}_x+(\rho^{(n)})^2v^{(n)}\|_{B_{2,1}^\frac{1}{2}}d\tau\Big).
\end{aligned}
\end{equation}

Recall that $u=(1-\partial_x^2)^{-1}m$ and $v=(1-\partial_x^2)^{-1}n$. Since $(1-\partial_x^2)^{-1}$ is a $\mathcal{S}^{-2}$ multiplier, Young inequality implies that
$$\|u\|_{B_{2,1}^\frac{1}{2}},\|u_x\|_{B_{2,1}^\frac{1}{2}},\|u_{xx}\|_{B_{2,1}^\frac{1}{2}} \le 2\|m\|_{B_{2,1}^\frac{1}{2}},$$
$$\|v\|_{B_{2,1}^\frac{1}{2}},\|v_x\|_{B_{2,1}^\frac{1}{2}},\|v_{xx}\|_{B_{2,1}^\frac{1}{2}} \le 2\|n\|_{B_{2,1}^\frac{1}{2}},$$
and with the fact that the critical Besov space $B_{2,1}^\frac{1}{2}$ is an algebra and embedding theory, we obtain
\begin{equation}
\begin{aligned}
&\|(u^{(n)}v^{(n)})_x\|_{B_{2,\infty}^\frac{1}{2}\bigcap L^\infty} \le C\|u^{(n)}v^{(n)}\|_{B_{2,1}^\frac{3}{2}} \le C\|m^{(n)}\|_{B_{2,1}^\frac{1}{2}}\|v^{(n)}\|_{B_{2,1}^\frac{1}{2}},\\
&\|-\rho^{(n)}(u^{(n)}_xv^{(n)}+u^{(n)}v^{(n)}_x)\|_{B_{2,1}^\frac{1}{2}} \le C\|\rho^{(n)}\|_{B_{2,1}^\frac{1}{2}}\|m^{(n)}\|_{B_{2,1}^\frac{1}{2}}\|n^{(n)}\|_{B_{2,1}^\frac{1}{2}},\\
&\|-3m^{(n)}u^{(n)}_xv^{(n)}-(\rho^{(n)})^2u^{(n)}\|_{B_{2,1}^\frac{1}{2}} \le C(\|{m^{(n)}}\|_{B_{2,1}^\frac{1}{2}}^2\|n^{(n)}\|_{B_{2,1}^\frac{1}{2}}+\|{\rho^{(n)}}\|_{B_{2,1}^\frac{1}{2}}^2\|{m^{(n)}}\|_{B_{2,1}^\frac{1}{2}}),\\
&\|-3n^{(n)}u^{(n)}v^{(n)}_x+(\rho^{(n)})^2v^{(n)}\|_{B_{2,1}^\frac{1}{2}} \le C(\|{n^{(n)}}\|_{B_{2,1}^\frac{1}{2}}^2\|m^{(n)}\|_{B_{2,1}^\frac{1}{2}}+\|{\rho^{(n)}}\|_{B_{2,1}^\frac{1}{2}}^2\|{n^{(n)}}\|_{B_{2,1}^\frac{1}{2}}).
\end{aligned}
\end{equation}

Take (5) into (4) yields that
\begin{equation}
\begin{aligned}
\|\rho^{(n+1)}\|_{L_t^\infty(B_{2,1}^\frac{1}{2})} &\le e^{CA_{n,s}}\Big(\|S_{n+1}\rho_0\|_{B_{2,1}^\frac{1}{2}}\\
&\ \ \ +C\int_0^te^{-CA_{n,s}(\tau)}\|\rho^{(n)}\|_{B_{2,1}^\frac{1}{2}}\|m^{(n)}\|_{B_{2,1}^\frac{1}{2}}\|n^{(n)}\|_{B_{p,r}^s}d\tau\Big),\\
\|m^{(n+1)}\|_{L_t^\infty(B_{2,1}^\frac{1}{2})} &\le e^{CA_{n,s}}\Big(\|S_{n+1}m_0\|_{B_{2,1}^\frac{1}{2}}\\
&\ \ \ +C\int_0^te^{-CA_{n,s}(\tau)}(\|{m^{(n)}}\|_{B_{2,1}^\frac{1}{2}}^2\|n^{(n)}\|_{B_{2,1}^\frac{1}{2}}+\|{\rho^{(n)}}\|_{B_{2,1}^\frac{1}{2}}^2\|{m^{(n)}}\|_{B_{2,1}^\frac{1}{2}})d\tau\Big),\\
\|n^{(n+1)}\|_{L_t^\infty(B_{2,1}^\frac{1}{2})} &\le e^{CA_{n,s}}\Big(\|S_{n+1}n_0\|_{B_{p,r}^s}\\
&\ \ \ +C\int_0^te^{-CA_{n,s}(\tau)}(\|{n^{(n)}}\|_{B_{2,1}^\frac{1}{2}}^2\|m^{(n)}\|_{B_{2,1}^\frac{1}{2}}+\|{\rho^{(n)}}\|_{B_{2,1}^\frac{1}{2}}^2\|{n^{(n)}}\|_{B_{2,1}^\frac{1}{2}})d\tau\Big).
\end{aligned}
\end{equation}

Denote $B_n(t)=\|\rho^{(n)}\|_{B_{2,1}^\frac{1}{2}}+\|m^{(n)}\|_{B_{2,1}^\frac{1}{2}}+\|n^{(n)}\|_{B_{2,1}^\frac{1}{2}}$ and take the summation of three inequalities in (6), we get

\begin{equation}
\begin{aligned}
\sup_{t\in [0,T)}B_{n+1}(t) \le e^{C\int_0^tB_n(\tau)^2d\tau}\Big(B(0)+C\int_0^te^{-C\int_0^\tau B_n(s)^2ds}B_n(\tau)^3d\tau\Big).
\end{aligned}
\end{equation}

With a similar argument in [XX], we have the upper bound of $B_{n+1}(t)$
\begin{equation}
B_{n+1}(t) \le \frac{B(0)}{\sqrt{1-4CB(0)^2 t}},
\end{equation}
which means $(\{\rho^{(n)},m^{(n)},n^{(n)}\})_{n\in\mathbb{N}}$ is uniformly bounded in $(C([0,T];B_{2,1}^{\frac{1}{2}}))^3$. Thanks to the structure of 3NS, it is easy to deduce that $(\{\partial_t\rho^{(n+1)},\partial_t m^{(n+1)},\partial_t n^{(n+1)}\})$ in $C([0,T);B_{2,1}^{\frac{1}{2}}))^{2n}$ is also uniformly bounded. Thus, $(M^n)_{n\in\mathbb{N}}$ is uniformly bounded in $[C([0,T);B_{2,1}^{\frac{1}{2}})\bigcap C^1([0,T);B_{2,1}^{-\frac{1}{2}})]^3$.

\textbf{Step 3.}
In this part, we will prove that $\{\rho^{(n)},m^{(n)},n^{(n)}\}$ is a Cauchy sequence in $C([0,T];B_{2,1}^{-\frac{1}{2}})^3$. We deduce from (3) that

\begin{equation}
\begin{small}
\begin{cases}
(\rho^{(n+m+1)}-\rho^{(n+1)})_t+u^{(n+m)}v^{(n+m)}(\rho^{(n+m+1)}-\rho^{(n+1)})_x=(u^{(n)}v^{(n)}-u^{(n+m)}v^{(n+m)})\rho^{(n+1)}_x+R_{n,m}^1,\\
(m^{(n+m+1)}-m^{(n+1)})_t+u^{(n+m)}v^{(n+m)}(m^{(n+m+1)}-m^{(n+1)})_x=(u^{(n)}v^{(n)}-u^{(n+m)}v^{(n+m)})m^{(n+1)}_x+R_{n,m}^2,\\
(n^{(n+m+1)}-n^{(n+1)})_t+u^{(n+m)}v^{(n+m)}(n^{(n+m+1)}-n^{(n+1)})_x=(u^{(n)}v^{(n)}-u^{(n+m)}v^{(n+m)})n^{(n+1)}_x+R_{n,m}^3,\\
(\rho^{(n+m+1)}-\rho^{(n+1)}, m^{(n+m+1)}-m^{(n+1)}, n^{(n+m+1)}-n^{(n+1)})=((S_{n+m+1}-S_{n+1})(\rho_0,m_0,n_0)),
\end{cases}
\end{small}
\end{equation}
where the remainders are
\begin{equation*}
\begin{aligned}
R_{n,m}^1&=\rho^{(n)}(u^{(n)}_xv^{(n)}+u^{(n)}v^{(n)}_x)-\rho^{(n+m)}(u^{(n+m)}_xv^{(n+m)}+u^{(n+m)}v^{(n+m)}_x),\\
R_{n,m}^2&=3m^{(n)}u^{(n)}_xv^{(n)}+(\rho^{(n)})^2u^{(n)}-3m^{(n+m)}u^{(n+m)}_xv^{(n+m)}-(\rho^{(n+m)})^2u^{(n+m)},\\
R_{n,m}^3&=3n^{(n)}u^{(n)}v^{(n)}_x-(\rho^{(n)})^2v^{(n)}-3n^{(n+m)}u^{(n+m)}v^{(n+m)}_x+(\rho^{(n+m)})^2v^{(n+m)}.
\end{aligned}
\end{equation*}

As we mentioned above, we cannot use Moser inequality directly to estimate the Besov norm of nonlinearities, which means
$$\|fg\|_{B_{2,r}^{-\frac{1}{2}}} \le C\|f\|_{B_{2,r}^\frac{1}{2}}\|g\|_{B_{2,r}^{-\frac{1}{2}}}, $$
is not holds for any $r \ge 1$. However, we have the following weak inequality for $r=+\infty$ [x],
$$\|fg\|_{B_{2,\infty}^{-\frac{1}{2}}} \le C\|f\|_{B_{2,\infty}^\frac{1}{2}\cap L^\infty}\|g\|_{B_{2,1}^{-\frac{1}{2}}}.$$
Applying Lemma 2.3 with $p=p_1=2$, $r=\infty$, we have
\begin{equation}
\begin{small}
\begin{aligned}
\|\rho^{(n+m+1)}-&\rho^{(n+1)}\|_{B_{2,\infty}^{-\frac{1}{2}}} \le \|S_{n+m+1}\rho_0-S_{n+1}\rho_0\|_{B_{2,\infty}^{-\frac{1}{2}}}+C(\int_0^T(V'(\tau)\|(\rho^{(n+m+1)}-\rho^{(n+1)})(\tau)\|_{B_{2,\infty}^{-\frac{1}{2}}}\\
&+\|(u^{(n)}v^{(n)}-u^{(n+m)}v^{(n+m)})\rho^{(n+1)}_x(\tau)\|_{B_{2,\infty}^{-\frac{1}{2}}}+\|R_{n,m}^1(\tau)\|_{B_{2,\infty}^{-\frac{1}{2}}})d\tau),\\
\|m^{(n+m+1)}-&m^{(n+1)}\|_{B_{2,\infty}^{-\frac{1}{2}}} \le\|S_{n+m+1}m_0-S_{n+1}m_0\|_{B_{2,\infty}^{-\frac{1}{2}}}+C(\int_0^TV'(\tau)\|(m^{(n+m+1)}-m^{(n+1)})(\tau)\|_{B_{2,\infty}^{-\frac{1}{2}}}\\
&+(\|(u^{(n)}v^{(n)}-u^{(n+m)}v^{(n+m)})m^{(n+1)}_x(\tau)\|_{B_{2,\infty}^{-\frac{1}{2}}}+\|R_{n,m}^2\|_{B_{2,\infty}^{-\frac{1}{2}}}(\tau))d\tau),\\
\|n^{(n+m+1)}-&n^{(n+1)}\|_{B_{2,\infty}^{-\frac{1}{2}}} \le\|S_{n+m+1}n_0-S_{n+1}n_0\|_{B_{2,\infty}^{-\frac{1}{2}}}+C(\int_0^TV'(\tau)\|(n^{(n+m+1)}-n^{(n+1)})(\tau)\|_{B_{2,\infty}^{-\frac{1}{2}}}\\
&+(\|(u^{(n)}v^{(n)}-u^{(n+m)}v^{(n+m)})n^{(n+1)}_x(\tau)\|_{B_{2,\infty}^{-\frac{1}{2}}}+\|R_{n,m}^3\|_{B_{2,\infty}^{-\frac{1}{2}}}(\tau))d\tau),\\
\end{aligned}
\end{small}
\end{equation}
where $V(t)=\displaystyle\int_0^t\|(u^{(n)}v^{(n)})_x(\tau)\|_{B_{2,1}^\frac{1}{2}}d\tau$.
Now we are going to estimate each terms in (10). With a directly computation, we have

\begin{equation}
\begin{aligned}
&\|(u^{(n)}v^{(n)}-u^{(n+m)}v^{(n+m)})\rho^{(n+1)}_x\|_{B_{2,\infty}^{-\frac{1}{2}}}&\\
\le &\|\rho_x^{(n+1)}\|_{B_{2,1}^{-\frac{1}{2}}}(\|u^{(n)}\|_{B_{2,\infty}^\frac{1}{2}\cap L^\infty}\|v^{(n+m)}-v^{(n)}\|_{B_{2,\infty}^\frac{1}{2}\cap L^\infty}+\|v^{(n+m)}\|_{B_{2,\infty}^\frac{1}{2}\cap L^\infty}\|u^{(n+m)}-u^{(n)}\|_{B_{2,\infty}^\frac{1}{2}\cap L^\infty})&\\
\le &\|\rho^{(n+1)}\|_{B_{2,1}^{\frac{1}{2}}}(\|m^{(n)}\|_{B_{2,1}^\frac{1}{2}}\|n^{(n+m)}-n^{(n)}\|_{B_{2,\infty}^{-\frac{1}{2}} }+\|n^{(n+m)}\|_{B_{2,1}^\frac{1}{2}}\|m^{(n+m)}-m^{(n)}\|_{B_{2,\infty}^{-\frac{1}{2}}})&\\
\le &C(\|m^{(n+m)}-m^{(n)}\|_{B_{2,\infty}^{-\frac{1}{2}}}+\|n^{(n+m)}-n^{(n)}\|_{B_{2,\infty}^{-\frac{1}{2}}}),&
\end{aligned}
\end{equation}
Similarly, we have
\begin{flalign}
&\|(u^{(n)}v^{(n)}-u^{(n+m)}v^{(n+m)})m^{(n+1)}_x\|_{B_{2,\infty}^{-\frac{1}{2}}} \le C(\|m^{(n+m)}-m^{(n)}\|_{B_{2,\infty}^{-\frac{1}{2}}}+\|n^{(n+m)}-n^{(n)}\|_{B_{2,\infty}^{-\frac{1}{2}}}),&\\
&\|(u^{(n)}v^{(n)}-u^{(n+m)}v^{(n+m)})n^{(n+1)}_x\|_{B_{2,\infty}^{-\frac{1}{2}}} \le C(\|m^{(n+m)}-m^{(n)}\|_{B_{2,\infty}^{-\frac{1}{2}}}+\|n^{(n+m)}-n^{(n)}\|_{B_{2,\infty}^{-\frac{1}{2}}}),&
\end{flalign}

For the remainders $R_{n,m}^1$, $R_{n,m}^2$ and $R_{n,m}^3$,

\begin{equation}
\begin{small}
\begin{aligned}
&\|R_{n,m}^1\|_{B_{2,\infty}^{-\frac{1}{2}}}\\
\le &\|\rho^{(m+n)}(u_x^{(m+n)}v^{(m+n)}+u^{(m+n)}v_x^{(m+n)}-u_x^{(n)}v^{(n)}-u^{(n)}v_x^{(n)})\|_{B_{2,\infty}^{-\frac{1}{2}}}\\
&+\|(\rho^{(m+n)}-\rho^{(n)})(u_x^{(n)}v^{(n)}+u^{(n)}v_x^{(n)})\|_{B_{2,\infty}^{-\frac{1}{2}}}\\
\le &\|\rho^{(m+n)}-\rho^{(n)}\|_{B_{2,1}^{-\frac{1}{2}}}\|(u_x^{(n)}v^{(n)}+u^{(n)}v_x^{(n)})\|_{B_{2,\infty}^\frac{1}{2}\cap L^\infty}+\|\rho^{(m+n)}\|_{B_{2,1}^{-\frac{1}{2}}}\|v^{(n+m)}(u_x^{(n+m)}-u_x^{(n)})\\
&+u_x^{(n)}(v^{(m+n)}-v^{(n)})+u^{(m+n)}(v_x^{(n+m)}-v_x^{(n)})+v_x^{(n)}(u^{(n+m)}-u^{(n)})\|_{B_{2,\infty}^\frac{1}{2}\cap L^\infty}\\
\le &C\Big(\|\rho^{(n+m)}-\rho^{(n)}\|_{B_{2,1}^{-\frac{1}{2}}}\|m^{(n)}\|_{B_{2,1}^\frac{1}{2}}\|n^{(n)}\|_{B_{2,1}^\frac{1}{2}}+\|\rho^{(n+m)}\|_{B_{2,1}^\frac{1}{2}}((\|n^{(n+m)}\|_{B_{2,1}^\frac{1}{2}}+\|n^{(n)}\|_{B_{2,1}^\frac{1}{2}})
\|m^{(n+m)}_x-m^{(n)}_x\|_{B_{2,1}^{-\frac{1}{2}}}\\
&+(\|m^{(n)}_x\|_{B_{2,1}^\frac{1}{2}}+\|m^{(n+m)}\|_{B_{2,1}^\frac{1}{2}})\|n^{(n+m)}_x-n^{(n)}_x\|_{B_{2,1}^{-\frac{1}{2}}}\Big)\\
\le &C(\|\rho^{(n+m)}-\rho^{(n)}\|_{B_{2,1}^{-\frac{1}{2}}}+\|m^{(n+m)}-m^{(n)}\|_{B_{2,1}^{-\frac{1}{2}}}+\|n^{(n+m)}-n^{(n)}\|_{B_{2,1}^{-\frac{1}{2}}}).
\end{aligned}
\end{small}
\end{equation}

With a similar technique, we deduce that $R_{n,m}^2$ and $R_{n,m}^2$ are also bounded by
\begin{flalign}
&\|R_{n,m}^2\|_{B_{2,\infty}^{-\frac{1}{2}}} \le C(\|\rho^{(n+m)}-\rho^{(n)}\|_{B_{2,1}^{-\frac{1}{2}}}+\|m^{(n+m)}-m^{(n)}\|_{B_{2,1}^{-\frac{1}{2}}}+\|n^{(n+m)}-n^{(n)}\|_{B_{2,1}^{-\frac{1}{2}}}),&\\
&\|R_{n,m}^3\|_{B_{2,\infty}^{-\frac{1}{2}}} \le C(\|\rho^{(n+m)}-\rho^{(n)}\|_{B_{2,1}^{-\frac{1}{2}}}+\|m^{(n+m)}-m^{(n)}\|_{B_{2,1}^{-\frac{1}{2}}}+\|n^{(n+m)}-n^{(n)}\|_{B_{2,1}^{-\frac{1}{2}}}).&
\end{flalign}

Take (11)-(16) into (10), and define $$E_{n,m,r}(t)=\|\rho^{(n+m)}-\rho^{(n)}\|_{B_{2,r}^{-\frac{1}{2}}}+\|m^{(n+m)}-m^{(n)}\|_{B_{2,r}^{-\frac{1}{2}}}+\|n^{(n+m)}-n^{(n)}\|_{B_{2,r}^{-\frac{1}{2}}},$$
then we can obtain that
\begin{equation*}
E_{n+1,m,\infty}(t) \le E_{n+1,m,\infty}(0)+C\int_0^t \Big(\|(m^{(n)}n^{(n)})(\tau)\|_{B_{2,1}^{-\frac{1}{2}}}E_{n+1,m,\infty}(\tau)+E_{n,m,1}(\tau)\Big)d\tau.
\end{equation*}
With Gronwall's inequality and denote $Z_{n+1,m,r}(t)=E_{n+1,m,r}(t)e^{-C\int_0^t \|(m^{(n)}n^{(n)})(\tau)\|_{B_{2,1}^{-\frac{1}{2}}}d\tau}$, we arrive at
\begin{equation}
Z_{n+1,m,\infty}(t) \le Z_{n+1,m,\infty}(0)+C\int_0^t Z_{n,m,1}(\tau)d\tau.
\end{equation}
We can choose $S>0$ large enough such that $Z<1$. In view of the following interpolation inequality [20]
\begin{equation}
\|f\|_{B_{2,1}^s} \le C\|f\|_{B_{2,\infty}^s}\log \left(e+\frac{\|f\|_{B_{2,\infty}^{s+1}}}{\|f\|_{B_{2,\infty}^s}}\right),
\end{equation}
and the auxiliary function
\begin{equation*}
q(x)=\log(e+\theta)\log\frac{1}{x}-\log(e+\frac{\theta}{x})+\log(e+\theta),\ \ \ \theta>0.
\end{equation*}
Since for $0<x \le 1$, $q'(x)<0$ and $q(1)=0$, we have
\begin{equation}
\log(e+\frac{\theta}{x}) \le \log(e+\theta)(1+\log\frac{1}{x}).
\end{equation}

Throw (18), we have
$$Z_{n,m,1} \le CZ_{n,m,\infty}(1-\log(Z_{n,m,\infty}).$$
Together with (17), we obtain
\begin{equation}
Z_{n+1,m,\infty}(t) \le Z_{n+1,m,\infty}(0)+C\int_0^t Z_{n,m,\infty}(1-\log(Z_{n,m,\infty})(\tau)d\tau.
\end{equation}

Since $S_j$ is low-frequency cut-off operator, which means for any function $h$,  $S_{n+m+1}h-S_{n+1}h=\displaystyle\sum_{j=n+1}^{n+m}\Delta_jh$, we claim that there exits a constant $C$ which only depend on T, such that
\begin{equation}
Z_{n+1,m,\infty}(t) \le C_T\Big(2^{-n}+\int_0^t Z_{n,m,\infty}(\tau)d\tau\Big).
\end{equation}

Using induction with respect to the index $n$, for a fixed number $m \in \mathbb{N}$, we get
\begin{equation}
Z_{n+1,m,\infty}(t))\left(1-\log Z_{n+1,m,\infty}(t)\right) \le C\displaystyle\sum_{k=0}^n 2^{-(n-k)}\frac{(TC)^k}{k!}+\frac{(TC)^{n+1}}{(n+1)!}Z_{0,m,\infty}(t)\left(1-\log Z_{0,m,\infty}(t)\right).
\end{equation}
Due to the definition of $Z$, for $\|m^{(n)}\|_{B_{2,1}^{-\frac{1}{2}}}$, $\|n^{(n)}\|_{B_{2,1}^{-\frac{1}{2}}}$ and $C>0$ are bounded independently of index $m$, $n$, then there exists a constant $\tilde{C}>0$ such that
$$\|\rho^{(n+m)}-\rho^{(n)}\|_{B_{2,\infty}^{-\frac{1}{2}}}+\|m^{(n+m)}-m^{(n)}\|_{B_{2,\infty}^{-\frac{1}{2}}}+\|n^{(n+m)}-n^{(n)}\|_{B_{2,\infty}^{-\frac{1}{2}}} \le \tilde{C}2^{-n}.$$
Thus $(\rho^{(n)},m^{(n)},n^{(n)})$ is a Cauchy sequence in $C([0,T);B_{2,\infty}^{-\frac{1}{2}})^3$, and take advantage of the interpolation inequality (18), we deduce that $(\rho^{(n)},m^{(n)},n^{(n)})$ tends to $(\rho,m,n)$ in $C([0,T);B_{2,1}^{-\frac{1}{2}})^3$.

\textbf{Step 4.}
It remains to prove that the solution $(\rho,m,n)$ of (2) is uniform and continuous. Suppose $\rho^1,m^1,n^1$ and $\rho^2,m^2,n^2$ are two solutions of system(2). Similarly with (9), we have
\begin{equation}
\begin{cases}
(\rho^{1}-\rho^{2})_t+u^{1}v^{1}(\rho^{1}-\rho^{2})_x=(u^{2}v^{2}-u^{1}v^{1})\rho^{2}_x+R^1,\\
(m^{1}-m^{2})_t+u^{1}v^{1}(m^{1}-m^{2})_x=(u^{2}v^{2}-u^{1}v^{1})m^{2}_x+R^2,\\
(n^{1}-n^{2})_t+u^{1}v^{1}(n^{1}-n^{2})_x=(u^{2}v^{2}-u^{1}v^{1})n^{2}_x+R^3,
\end{cases}
\end{equation}
where
\begin{equation*}
\begin{cases}
R^1=\rho^{2}(u^{2}_xv^{2}+u^{2}v^{2}_x)-\rho^{1}(u^{1}_xv^{1}+u^{1}v^{1}_x),\\
R^2=3m^{2}u^{2}_xv^{2}+(\rho^{2})^2u^{2}-3m^{1}u^{1}_xv^{1}-(\rho^{1})^2u^{1},\\
R^3=3n^{2}u^{2}v^{2}_x-(\rho^{2})^2v^{2}-3n^{1}u^{1}v^{1}_x+(\rho^{1})^2v^{1}.\\
\end{cases}
\end{equation*}
For convenience, we denote
$$\|M(t)\|_X=e^{-c\int_0^t\|(m^1n^1)(\tau)|\|_{B_{2,1}^\frac{1}{2}}d\tau}\Big(\|(\rho^1-\rho^2)(t)\|_X+\|(m^1-m^2)(t)\|_X+\|(n^1-n^2)(t)\|_X\Big)$$
in any function space $X$. Similarly with Step 3, by virtue of Lemma 2.3, we have the following estimation
\begin{equation}
\|M(t)\|_{B_{2,\infty}^{-\frac{1}{2}}} \le\|M(0)\|_{B_{2,\infty}^{-\frac{1}{2}}}+\int_0^t\|M(\tau)\|_{B_{2,1}^{-\frac{1}{2}}}d\tau.
\end{equation}
We can choose constant $c$ large enough such that $\|M(t)\|_{B_{2,\infty}^{-\frac{1}{2}}} < 1$. Thanks to interpolation inequality (18), we have
\begin{equation}
\begin{aligned}
\|M(t)\|_{B_{2,\infty}^{-\frac{1}{2}}}&\le \|M(0)\|_{B_{2,\infty}^{-\frac{1}{2}}}+\int_0^t\|M(t)\|_{B_{2,\infty}^{-\frac{1}{2}}}(\tau)\log\left(e+\frac{\|M(\tau)\|_{B_{2,\infty}^\frac{1}{2}}}{\|M(\tau)\|_{B_{2,\infty}^{-\frac{1}{2}}}}\right)d\tau\\
&\le \|M(0)\|_{B_{2,\infty}^{-\frac{1}{2}}}+C\int_0^t\|M(\tau)\|_{B_{2,\infty}^{-\frac{1}{2}}}\left(1-\|M(\tau)\|_{B_{2,\infty}^{-\frac{1}{2}}}(\tau))\right)d\tau.
\end{aligned}
\end{equation}
By virtue of Osgood Lemma 2.4, we take $\mu(r)=r(1-\log r)$, $\mathcal{N}(r)=\log(1-\log(r))$. Finally, we obtain
\begin{equation}
\|M(t)\|_{B_{2,\infty}^{-\frac{1}{2}}} \le e^{1-[1-\log \|M(0)\|_{B_{2,\infty}^{-\frac{1}{2}}}]e^{-ct}}.
\end{equation}
Thus, by definition of $\|M\|_X$, we conclude that,
\begin{equation*}
\begin{aligned}
& \|(\rho^1-\rho^2)(t)\|_{B_{2,\infty}^{-\frac{1}{2}}}+\|(m^1-m^2)(t)\|_{B_{2,\infty}^{-\frac{1}{2}}}+\|(n^1-n^2)(t)\|_{B_{2,\infty}^{-\frac{1}{2}}}\\
\le & C_T\Big(\|(\rho^1-\rho^2)(0)\|_{B_{2,\infty}^{-\frac{1}{2}}}+\|(m^1-m^2)(0)\|_{B_{2,\infty}^{-\frac{1}{2}}}+\|(n^1-n^2)(0)\|_{B_{2,\infty}^{-\frac{1}{2}}}\Big).
\end{aligned}
\end{equation*}
By virtue of Lemma 2.1, we have just proved that
\begin{equation*}
\begin{aligned}
& \|(\rho^1-\rho^2)(t)\|_{B_{2,\infty}^s}+\|(m^1-m^2)(t)\|_{B_{2,\infty}^s}+\|(n^1-n^2)(t)\|_{B_{2,\infty}^s}\\
\le & C_T\Big(\|(\rho^1-\rho^2)(0)\|_{B_{2,\infty}^s}+\|(m^1-m^2)(0)\|_{B_{2,\infty}^s}+\|(n^1-n^2)(0)\|_{B_{2,\infty}^s}\Big),
\end{aligned}
\end{equation*}
holds for any $s>\frac{1}{2}$, together with interpolation
$$\|f\|_{B_{2,1}^\frac{1}{2}} \le C_\theta \|f\|_{B_{2,\infty}^s}^\theta\|f\|_{B_{2,\infty}^{-\frac{1}{2}}}^{1-\theta},$$
and take $\theta=\dfrac{1}{s+1} \in (0,1)$, we prove the uniqueness and H\"{o}lder continuity. Hence we complete the proof.

\end{proof}

\begin{remark}
By applying the $S^{-2}$ operator $(1-\partial_x^2)^{-1}$ to 3NS, we can rewrite it into convolution form
\begin{equation}
\begin{cases}
\rho_t+(uv\rho )_x=0,\\
u_t+uu_xv+g(x)*(3uu_xv+u_x^2v_x+uu_{xx}v_x+\rho^2 u)+\partial_x g(x)*uu_xv_x=0,\\
v_t+uvv_x+g(x)*(3uvv_x+u_xv_x^2+u_xvv_{xx}-\rho^2 v)+\partial_x g(x)*u_xvv_x=0,
\end{cases}
\end{equation}
where $g(x)$ is the green function of $(1-\partial_x^2)^{-1}$. Thus in the convolution sense, the critical Besov space should be $B_{2,1}^\frac{3}{2}$.
\end{remark}

\begin{remark}
As Guo, Liu, Molinet and Yin [21] have just proved Camassa-Holm equation, Degasperis-Procesi equation and Novikov equation are ill-posed in $B_{p,r}^\frac{1}{p}$ with any $r>1$, thus the critical Besov space $B_{2,1}^\frac{1}{2}$ is almost the best well-posed space.
\end{remark}

\section{Exponential decay of solution}
In this section, we consider the exponential decay of solution to 3NS. We prove that the solutions will retain the corresponding properties at infinity as the initial data decay exponentially. Let's focus on solution $(\rho,m,n)$ at first.
\begin{theorem}
Let $(\rho_0,m_0,n_0) \in (H^s)^3$ with $s>\frac{1}{2}$, $T>0$ is the maximal existence time. If there exists some constant $a>0$, such that the initial data satisfy
\begin{equation*}
\rho_0(x),m_0(x),n_0(x) \sim \mathcal{O}(e^{-a|x|}), \ \ \ as \ \ \ |x| \rightarrow \infty,
\end{equation*}
then for any $t \in [0,T)$, it follows that
\begin{equation*}
\rho(t,x),m(t,x),n(t,x) \sim \mathcal{O}(e^{-a|x|}), \ \ \ as \ \ \ |x| \rightarrow \infty.
\end{equation*}
\end{theorem}

\begin{proof}
Without loss of generality, we only prove the case of $x\rightarrow +\infty$. Multiplying 3NS$_1$ (the first equation of 3NS) by $e^{ax}$, we have
\begin{equation*}
(e^{ax}\rho)_t+(e^{ax}\rho)(u_xv+uv_x)+(e^{ax}\rho_x)uv=0.
\end{equation*}
Taking inner product with $(e^{ax}\rho)^{2k-1}$ $(k\in\mathbb{N}^*)$, we have
\begin{equation}
\begin{aligned}
\int_\mathbb{R}(e^{ax}\rho)^{2k-1}(e^{ax}\rho)_tdx=&-\int_\mathbb{R}uv(e^{ax}\rho)^{2k-1}[(e^{ax}\rho)_x-ae^{ax}\rho]dx\\
&-\int_\mathbb{R} (u_xv+uv_x)(e^{ax}\rho)^{2k}dx.
\end{aligned}
\end{equation}
Taking advantage of H\"{o}lder inequality, we have
\begin{equation}
\int_\mathbb{R} (e^{ax}\rho)^{2k-1}(e^{ax}\rho)_tdx=\|e^{ax}\rho\|_{L^{2k}}^{2k-1}\frac{d}{dt}\|e^{ax}\rho\|_{L^{2k}},
\end{equation}
\begin{equation}
\begin{aligned}
-\int_\mathbb{R}uv(e^{ax}\rho)^{2k-1}(e^{ax}\rho)_xdx=&\frac{1}{2k}\int_\mathbb{R}(u_xv+uv_x)(e^{ax}\rho)^{2k}dx\\
\le &(\frac{1}{2k}(\|u_x\|_{L^\infty}\|v\|_{L^\infty}+\|u\|_{L^\infty}\|v_x\|_{L^\infty})\|e^{ax}\rho\|_{L^{2k}}^{2k},
\end{aligned}
\end{equation}
and
\begin{equation}
\int_\mathbb{R}auv(e^{ax}\rho)^{2k}dx\le a\|u\|_{L^\infty}\|v\|_{L^\infty}\|e^{ax}\rho\|_{L^{2k}}^{2k}.
\end{equation}

Taking (29)-(31) back into (28) yields
\begin{equation}
\frac{d}{dt}\|e^{ax}\rho\|_{L^{2k}}\le \Big[(1+\frac{1}{2k})(\|u_x\|_{L^\infty}\|v\|_{L^\infty}+\|u\|_{L^\infty}\|v_x\|_{L^\infty})+a\|u\|_{L^\infty}\|v\|_{L^\infty}\Big]\|e^{ax}\rho\|_{L^{2k}}.
\end{equation}
Multiplying 3NS$_2$ by $e^{ax}$ and taking inner product with $(e^{ax}m)^{2k-1}$, we have
\begin{equation}
\begin{aligned}
\int_\mathbb{R}(e^{ax}m)^{2k-1}(e^{ax}m)_tdx=&-\int_\mathbb{R}uv(e^{ax}m)^{2k-1}[(e^{ax}m)_x-ae^{ax}m]dx\\
&-\int_\mathbb{R} 3u_xv(e^{ax}m)^{2k}dx-\int_\mathbb{R}(e^{ax}m)^{2k-1}(e^{ax}\rho)\rho udx.
\end{aligned}
\end{equation}
Similarly with the estimation of 3NS$_1$, we obtain
\begin{equation}
\begin{aligned}
\frac{d}{dt}\|e^{ax}m\|_{L^{2k}}\le & \Big[(3+\frac{1}{2k})(\|u_x\|_{L^\infty}\|v\|_{L^\infty}+\frac{1}{2k}\|u\|_{L^\infty}\|v_x\|_{L^\infty})+a\|u\|_{L^\infty}\|v\|_{L^\infty}\Big]\|e^{ax}m\|_{L^{2k}}\\
&+\|\rho\|_{L^\infty}\|u\|_{L^\infty}\|e^{ax}\rho\|_{L^{2k}}.
\end{aligned}
\end{equation}
We claim that for $\|e^{ax}n\|_{L^{2k}}$ we have
\begin{equation}
\begin{aligned}
\frac{d}{dt}\|e^{ax}n\|_{L^{2k}}\le & \Big[(3+\frac{1}{2k})(\|u\|_{L^\infty}\|v_x\|_{L^\infty}+\frac{1}{2k}\|u_x\|_{L^\infty}\|v\|_{L^\infty})+a\|u\|_{L^\infty}\|v\|_{L^\infty}\Big]\|e^{ax}n\|_{L^{2k}}\\
&+\|\rho\|_{L^\infty}\|v\|_{L^\infty}\|e^{ax}\rho\|_{L^{2k}}.
\end{aligned}
\end{equation}
Denote $F_k(t)=\|e^{ax}\rho(t,\cdot)\|_{L^{2k}}+\|e^{ax}m(t,\cdot)\|_{L^{2k}}+\|e^{ax}n(t,\cdot)\|_{L^{2k}}$ and define the uniform bound
$$M\triangleq\sup_{t\in[0,T)}(\|\rho(t,\cdot)\|_{H^s}+\|m(t,\cdot)\|_{H^s}+\|n(t,\cdot)\|_{H^s}).$$
Due to the Sobolev embedding theorem, there exists $C=C_{(M,a,\frac{1}{k})}$ such that
\begin{equation}
\frac{d}{dt}F_k(t) \le C F_k(t).
\end{equation}

Note that for any function $f\in L^1 \bigcap L^\infty$ with $f \in L^p$ for all $p>1$, the following limit holds
$$\lim_{p\rightarrow \infty}\|f\|_{L^p}=\|f\|_{L^\infty}.$$
Taking the limit as $k\rightarrow \infty$ in (36) and applying Young's inequality yields
$$F_\infty(t) \le F_\infty(0)e^{C_{(M,a)}T},$$
which means
$$|e^{ax}\rho(t,x)|+|e^{ax}m(t,x)|+|e^{ax}n(t,x)| \le e^{CT}\Big(\|e^{ax}\rho_0(x)\|_{L^\infty}+\|e^{ax}m_0(x)\|_{L^\infty}+\|e^{ax}_0(x)\|_{L^\infty}\Big).$$
Therefore, we complete the proof of Theorem 3.1.
\end{proof}

The following theorem illustrates exponential decay of $u$ and $v$. Note that here are convolutions in (27), thus we should offer an useful lemma first.
\begin{lemma}[22]
Assume the function $g(x)=\frac{1}{2}e^{-|x|}$. Let the weighted function $\varphi_N(x)$ be
\begin{equation*}
\varphi_N(x)=\begin{cases}
e^{\alpha N},&x \in (-\infty,-N),\\
e^{-\alpha x},&x \in [-N,0],\\
1,&x \in (0,\infty),
\end{cases}
\end{equation*}
where $N \in \mathbb{N}^*$. If the constant $\alpha \in (0,1)$, then

(i) $0 \le \varphi_N'(x) \le \varphi_N(x)$ a.e.,

(ii)there exists $C_0$, such that $\varphi_N(x)(g*(\varphi_N)^{-1})(x) \le C_0$, $\varphi_N(x)(\partial_xg*(\varphi_N)^{-1})(x) \le C_0.$
\end{lemma}

Now we establish the exponential decay of strong solution $u$ and $v$.
\begin{theorem}
Let $(\rho_0,u_0,v_0) \in H^s \times H^s \times H^{s-1}$ with $s>\frac{5}{2}$, $T>0$ be the maximal existence time, and the corresponding solution $(\rho,u,v) \in C([0,T);H^s \times H^s \times H^{s-1})$. If there exists some constant $\alpha \in (0,1)$ such that the initial data satisfy
\begin{equation*}
\begin{cases}
\ |\rho_0(x)|,\ |u_0(x)|,\ |v_0(x)|\ \sim \mathcal{O}(e^{\alpha x}),&x \downarrow -\infty,\\
|\rho_{0x}(x)|,|u_{0x}(x)|,|v_{0x}(x)| \sim \mathcal{O}(e^{\alpha x}),&x \downarrow -\infty,\\
\end{cases}
\end{equation*}
then the solutions $(\rho,u,v)$ is exponential decay as
\begin{equation*}
\begin{cases}
\ |\rho(t,x)|,\ |u(t,x)|,\ |v(t,x)|\ \sim \mathcal{O}(e^{\alpha x}),&x \downarrow -\infty,\\
|\rho_{x}(t,x)|,|u_{x}(t,x)|,|v_{x}(t,x)| \sim \mathcal{O}(e^{\alpha x}),&x \downarrow -\infty,\\
\end{cases}
\end{equation*}
uniformly in the time interval $[0,T)$.
\end{theorem}

\begin{proof}
Similarly with (32), multiplying 3NS$_1$ by $\varphi_N$ and taking inner product with $(\varphi_N\rho)^{2k-1}$, we have
\begin{equation}
\frac{d}{dt}\|\varphi_N\rho\|_{L^{2k}}\le \Big[(1+\frac{1}{2k})(\|u_x\|_{L^\infty}\|v\|_{L^\infty}+\|u\|_{L^\infty}\|v_x\|_{L^\infty})+\|u\|_{L^\infty}\|v\|_{L^\infty}\Big]\|\varphi_N\rho\|_{L^{2k}}.
\end{equation}
Differentiating 3NS$_1$ with respect to $x$ and multiplying by $\varphi_N$, we get
\begin{equation}
(\varphi_N\rho)_t+(\varphi_N\rho_{xx})uv+2(\varphi_N\rho_x)(uv)_x+(\varphi_N\rho)(uv)_{xx}=0,
\end{equation}
and taking inner product with $(\varphi_N\rho_x)^{2k-1}$ yields
\begin{equation}
\begin{aligned}
\int_\mathbb{R}(\varphi_N\rho_x)^{2k-1}(\varphi_N\rho_x)_tdx=&-\int_\mathbb{R}uv(\varphi_N\rho_x)^{2k-1}\varphi_N\rho_{xx}-2\int_\mathbb{R}(\varphi_N\rho_x)^{2k}(u_xv+uv_x)dx\\
&-\int_\mathbb{R} (\varphi_N\rho_x)^{2k-1}(\varphi_N\rho)(u_{xx}v+2u_xv_x+uv_{xx})dx.
\end{aligned}
\end{equation}
Note that
\begin{equation}
\int_\mathbb{R}(\varphi_N\rho_x)^{2k-1}(\varphi_N\rho_x)_tdx =\|\varphi_N\rho_x\|_{L^{2k}}^{2k-1}\frac{d}{dt}\|\varphi_N\rho_x\|_{L^{2k}},
\end{equation}
and by using H\"{o}lder inequality, we obtain
\begin{equation}
\begin{aligned}
-\int_\mathbb{R}uv(\varphi_N\rho_x)^{2k-1}&\varphi_N\rho_{xx}=-\int_\mathbb{R}uv[(\varphi_N\rho_x)_x-\varphi_N'(x)\rho_x]dx\\
=&\frac{1}{2k}\int_\mathbb{R}(u_x+uv_x)(\varphi_N\rho_x)^{2k}dx+\int_\mathbb{R}uv(\varphi_N\rho_x)^{2k-1}(\varphi_N'(x)\rho_x)dx\\
\le& \Big[\frac{1}{2k}(\|u_x\|_{L^\infty}\|v\|_{L^\infty}+\|u\|_{L^\infty}\|v_x\|_{L^\infty})+\|u\|_{L^\infty}\|v\|_{L^\infty}\Big]\|\varphi_N\rho_x\|_{L^{2k}}^{2k},
\end{aligned}
\end{equation}
\begin{equation}
-2\int_\mathbb{R}(\varphi_N\rho_x)^{2k}(u_xv+uv_x)dx \le 2(\|u_x\|_{L^\infty}\|v\|_{L^\infty}+\|u\|_{L^\infty}\|v_x\|_{L^\infty})\|\varphi_N\rho_x\|_{L^{2k}}^{2k},
\end{equation}
and
\begin{equation}
\begin{aligned}
-\int_\mathbb{R}&(\varphi_N\rho_x)^{2k-1}(\varphi_N\rho)(u_{xx}v+2u_xv_x+uv_{xx})dx\\
&\le (\|u_{xx}\|_{L^\infty}\|v\|_{L^\infty}+2\|u_{x}\|_{L^\infty}\|v_{x}\|_{L^\infty}+\|u\|_{L^\infty}\|v_{xx}\|_{L^\infty})\|\varphi_N\rho_x\|_{L^{2k}}^{2k-1}\|\varphi_N\rho\|_{L^{2k}}.
\end{aligned}
\end{equation}
Taking (40)-(43) into (39), we have
\begin{equation}
\begin{aligned}
\frac{d}{dt}\|\varphi_N\rho_x\|_{L^{2k}} &\le \Big[(2+\frac{1}{2k})(\|u_x\|_{L^\infty}\|v\|_{L^\infty}+\|u\|_{L^\infty}\|v_x\|_{L^\infty})+\|u\|_{L^\infty}\|v\|_{L^\infty}\Big]\|\varphi_N\rho_x\|_{L^{2k}}\\
&+ (\|u_{xx}\|_{L^\infty}\|v\|_{L^\infty}+2\|u_{x}\|_{L^\infty}\|v_{x}\|_{L^\infty}+\|u\|_{L^\infty}\|v_{xx}\|_{L^\infty})\|\varphi_N\rho\|_{L^{2k}}.
\end{aligned}
\end{equation}

For 3NS$_2$, multiplying $\varphi_N$ and taking inner product with $(\varphi_Nu)^{2k-1}$, we have
\begin{equation}
\begin{aligned}
\int_\mathbb{R}(\varphi_Nu)^{2k-1}(\varphi_Nu)_tdx=&-\int_\mathbb{R}(\varphi_Nu)^{2k}u_xvdx-\int_\mathbb{R}(\varphi_Nu)^{2k-1}[\varphi_N(g*F)]dx\\
&-\int_\mathbb{R}(\varphi_Nu)^{2k-1}[\varphi_N(\partial_xg*uu_xv_x)]dx,
\end{aligned}
\end{equation}
where $F=3uu_xv+u_x^2v_x+uu_{xx}v_x+\rho^2u$. Taking advantage of H\"{o}lder inequality again, we obtain
\begin{equation}
\begin{aligned}
\int_\mathbb{R}(\varphi_Nu)^{2k-1}&\Big[\varphi_N(g*F)+\varphi_N(\partial_xg*uu_xv_x)\Big]dx\\
&\le \|\varphi_Nu\|_{L^{2k}}^{2k-1}\Big(\|\varphi_N(g*F)_{L^{2k}}+\|\varphi_N(\partial_xg*uu_xv_x)\|_{L^{2k}}\Big),
\end{aligned}
\end{equation}
which leads the following inequality
\begin{equation}
\frac{d}{dt}\|\varphi_Nu\|_{L^{2k}} \le \|u_x\|_{L^\infty}\|v\|_{L^\infty}\|\varphi_Nu\|_{L^{2k}}+\|\varphi_N(g*F)\|_{L^{2k}}+\|\varphi_N(\partial_xg*uu_xv_x)\|_{L^{2k}}.
\end{equation}

Differentiating 3NS$_2$ with respect to $x$ and multiplying by $\varphi_N$, we get
\begin{equation}
(\varphi_Nu_x)_t+(\varphi_Nu_{xx})uv+(\varphi_Nu_x)u_xv+\varphi_N(\partial_xg*F)+\varphi_N(g*uu_xv_x)=0,
\end{equation}
here we used the fact $f+\partial_x^2g*f=g*f$. With a similar argument, it is easy to check that
\begin{equation}
\begin{aligned}
\frac{d}{dt}\|\varphi_Nu_x\|_{L^{2k}} \le &\Big[(1+\frac{1}{2k})\|u_x\|_{L^\infty}\|v\|_{L^\infty}+\frac{1}{2k}\|u\|_{L^\infty}\|v_x\|_{L^\infty}+\|u\|_{L^\infty}\|v\|_{L^\infty}\Big]\|\varphi_Nu_x\|_{L^{2k}}\\
&+\|\varphi_N(\partial_xg*F)\|_{L^{2k}}+\|\varphi_N(g*uu_xv_x)\|_{L^{2k}}.
\end{aligned}
\end{equation}
For $\|\varphi_Nv\|_{L^{2k}}$ and $\|\varphi_Nv_x\|_{L^{2k}}$, we claim
\begin{equation}
\frac{d}{dt}\|\varphi_Nv\|_{L^{2k}} \le \|u\|_{L^\infty}\|v_x\|_{L^\infty}\|\varphi_Nv\|_{L^{2k}}+\|\varphi_N(g*G)\|_{L^{2k}}+\|\varphi_N(\partial_xg*u_xvv_x)\|_{L^{2k}},
\end{equation}
\begin{equation}
\begin{aligned}
\frac{d}{dt}\|\varphi_Nv_x\|_{L^{2k}} \le &\Big[(1+\frac{1}{2k})\|u\|_{L^\infty}\|v_x\|_{L^\infty}+\frac{1}{2k}\|u_x\|_{L^\infty}\|v\|_{L^\infty}+\|u\|_{L^\infty}\|v\|_{L^\infty}\Big]\|\varphi_Nv_x\|_{L^{2k}}\\
&+\|\varphi_N(\partial_xg*G)\|_{L^{2k}}+\|\varphi_N(g*u_xvv_x)\|_{L^{2k}},
\end{aligned}
\end{equation}
with $G=3uvv_x+u_xv_x^2+u_xvv_{xx}-\rho^2v$.

Now, taking the limit $k \rightarrow =\infty$ in (37),(44),(47),(49),(50) and (51), we get the following inequalities
\begin{equation}
\begin{small}
\begin{aligned}
\frac{d}{dt}\|\varphi_N\rho\|_{L^\infty} \le & \Big[\|u_x\|_{L^\infty}\|v\|_{L^\infty}+\|u\|_{L^\infty}\|v_x\|_{L^\infty}+\|u\|_{L^\infty}\|v\|_{L^\infty}\Big]\|\varphi_N\rho\|_{L^\infty},\\
\frac{d}{dt}\|\varphi_N\rho_x\|_{L^\infty} \le & \Big[2(\|u_x\|_{L^\infty}\|v\|_{L^\infty}+\|u\|_{L^\infty}\|v_x\|_{L^\infty})+\|u\|_{L^\infty}\|v\|_{L^\infty}\Big]\|\varphi_N\rho_x\|_{L^\infty}\\
&+(\|u_{xx}\|_{L^\infty}\|v\|_{L^\infty}+2\|u_{x}\|_{L^\infty}\|v_{x}\|_{L^\infty}+\|u\|_{L^\infty}\|v_{xx}\|_{L^\infty})\|\varphi_N\rho\|_{L^\infty},\\
\frac{d}{dt}\|\varphi_Nu\|_{L^\infty} \le & \|u_x\|_{L^\infty}\|v\|_{L^\infty}\|\varphi_Nu\|_{L^{2k}}+\|\varphi_N(g*F)\|_{L^\infty}+\|\varphi_N(\partial_xg*uu_xv_x)\|_{L^\infty},\\
\frac{d}{dt}\|\varphi_Nu_x\|_{L^\infty} \le & \Big[\|u_x\|_{L^\infty}\|v\|_{L^\infty}+\|u\|_{L^\infty}\|v\|_{L^\infty}\Big]\|\varphi_Nu_x\|_{L^\infty}+\|\varphi_N(\partial_xg*F)\|_{L^\infty}+\|\varphi_N(g*uu_xv_x)\|_{L^\infty},\\
\frac{d}{dt}\|\varphi_Nv\|_{L^\infty} \le & \|u\|_{L^\infty}\|v_x\|_{L^\infty}\|\varphi_Nv\|_{L^\infty}+\|\varphi_N(g*G)\|_{L^{2k}}+\|\varphi_N(\partial_xg*u_xvv_x)\|_{L^\infty},\\
\frac{d}{dt}\|\varphi_Nv_x\|_{L^\infty} \le & \Big[\|u\|_{L^\infty}\|v_x\|_{L^\infty}+\|u\|_{L^\infty}\|v\|_{L^\infty}\Big]\|\varphi_Nv_x\|_{L^\infty}+\|\varphi_N(\partial_xg*G)\|_{L^\infty}+\|\varphi_N(g*u_xvv_x)\|_{L^\infty}.
\end{aligned}
\end{small}
\end{equation}
It remains to estimate the convolution parts. By virtue of Lemma 3.2, we have
\begin{equation}
\begin{aligned}
&\|\varphi_N(g*F)\|_{L^\infty}=\Big|\varphi_N(x)\int_\mathbb{R}\frac{1}{2}e^{-|x-y|}\frac{1}{\varphi_N(y)}\varphi_N(y)\Big(3uu_xv+u_x^2v_x+uu_{xx}v_x+\rho^2u)\Big)(y)dy\Big|\\
\le &C_0\Big[\Big(3\|u_x\|_{L^\infty}\|v\|_{L^\infty}+\|u_{xx}\|_{L^\infty}\|v_x\|_{L^\infty}+\|\rho\|_{L^\infty}^2\Big)\|\varphi_Nu\|_{L^\infty}+\|u_x\|_{L^\infty}\|v_x\|_{L^\infty}\|\varphi_Nu_x\|_{L^\infty}\Big].
\end{aligned}
\end{equation}
Thus, with a similar computation, each $L^\infty$ norm of convolution are bounded by
\begin{equation}
\begin{aligned}
\|\varphi_N(\partial_xg*F)\|_{L^\infty}\le &C_0\Big[\Big(3\|u_x\|_{L^\infty}\|v\|_{L^\infty}+\|u_{xx}\|_{L^\infty}\|v_x\|_{L^\infty}+\|\rho\|_{L^\infty}^2\Big)\|\varphi_Nu\|_{L^\infty}\\
&+\|u_x\|_{L^\infty}\|v_x\|_{L^\infty}\|\varphi_Nu_x\|_{L^\infty}\Big],
\end{aligned}
\end{equation}

\begin{equation}
\begin{aligned}
\|\varphi_N(g*G)\|_{L^\infty},\|\varphi_N&(\partial_xg*G)\|_{L^\infty}\le C_0\Big[\|u_x\|_{L^\infty}\|v_x\|_{L^\infty}\|\varphi_Nv_x\|_{L^\infty}\\
&+\Big(3\|u\|_{L^\infty}\|v_x\|_{L^\infty}+\|u_x\|_{L^\infty}\|v_{xx}\|_{L^\infty}+\|\rho\|_{L^\infty}^2\Big)\|\varphi_Nv\|_{L^\infty}\Big],\\
\end{aligned}
\end{equation}

\begin{equation}
\|\varphi_N(g*uu_xv_x)\|_{L^\infty},\|\varphi_N(\partial_xg*uu_xv_x)\|_{L^\infty} \le C_0\|u_x\|_{L^\infty}\|v_x\|_{L^\infty}\|\varphi_Nu\|_{L^\infty},\ \ \ \ \ \
\end{equation}
\begin{equation}
\|\varphi_N(g*u_xvv_x)\|_{L^\infty},\|\varphi_N(\partial_xg*u_xvv_x)\|_{L^\infty} \le C_0\|u_x\|_{L^\infty}\|v_x\|_{L^\infty}\|\varphi_Nv\|_{L^\infty}.\ \ \ \ \ \
\end{equation}

Taking (53)-(57) into (52), and denote $$A_N(t)=\|\varphi_N\rho\|_{L^\infty}+\|\varphi_N\rho_x\|_{L^\infty}+\|\varphi_Nu\|_{L^\infty}+\|\varphi_Nu_x\|_{L^\infty}+\|\varphi_Nv\|_{L^\infty}+\|\varphi_Nv_x\|_{L^\infty}.$$
Thanks to Sobolev's embedding theorem and the definition of $M$ in the proof of theorem 3.3, we get
\begin{equation*}
\frac{d}{dt}A_N(t) \le CA_N(t),
\end{equation*}
where $C=C(C_0,M)$. By applying Gronwall's inequality and taking limit $N\rightarrow +\infty$, we finally obtain that there exists some constant $\tilde{C}=\tilde{C}(C_0,M,T)>0$ such that
\begin{equation*}
\begin{aligned}
&\|e^{-\alpha x}\rho\|_{L^\infty}+\|e^{-\alpha x}\rho_x\|_{L^\infty}+\|e^{-\alpha x}u\|_{L^\infty}+\|e^{-\alpha x}u_x\|_{L^\infty}+\|e^{-\alpha x}v\|_{L^\infty}+\|e^{-\alpha x}v_x\|_{L^\infty}\\
\le&\tilde{C}\Big(\|e^{-\alpha x}\rho_0\|_{L^\infty}+\|e^{-\alpha x}\rho_{0x}\|_{L^\infty}+\|e^{-\alpha x}u_0\|_{L^\infty}+\|e^{-\alpha x}u_{0x}\|_{L^\infty}+\|e^{-\alpha x}v_0\|_{L^\infty}+\|e^{-\alpha x}v_{0x}\|_{L^\infty}\Big).
\end{aligned}
\end{equation*}
\end{proof}

Theorem 3.3 shows the exponential decay of solutions in case of $x \downarrow -\infty$. For $x \uparrow +\infty$, just choose another weighted function $\tilde{\varphi}_N(x)$ for $\alpha \in (0,1)$ as
\begin{equation*}
\tilde{\varphi}_N(x)\begin{cases}
1,&x \in (-\infty,0),\\
e^{\alpha x},&x \in [-0,N],\\
e^{\alpha N},&x \in (N,\infty),
\end{cases}
\end{equation*}
we have the following result.
\begin{corollary}
Let $(\rho_0,u_0,v_0) \in H^s \times H^s \times H^{s-1}$ with $s>\frac{5}{2}$, $T>0$ be the maximal existence time, and the corresponding solution $(\rho,u,v) \in C([0,T);H^s \times H^s \times H^{s-1})$. If there exists some constant $\alpha \in (0,1)$ such that the initial data satisfy
\begin{equation*}
\begin{cases}
\ |\rho_0(x)|,\ |u_0(x)|,\ |v_0(x)|\ \sim \mathcal{O}(e^{-\alpha x}),&x \uparrow \infty,\\
|\rho_{0x}(x)|,|u_{0x}(x)|,|v_{0x}(x)| \sim \mathcal{O}(e^{-\alpha x}),&x \uparrow \infty,\\
\end{cases}
\end{equation*}
then the solutions $(\rho,u,v)$ is exponential decay as
\begin{equation*}
\begin{cases}
\ |\rho(t,x)|,\ |u(t,x)|,\ |v(t,x)|\ \sim \mathcal{O}(e^{-\alpha x}),&x \uparrow \infty,\\
|\rho_{x}(t,x)|,|u_{x}(t,x)|,|v_{x}(t,x)| \sim \mathcal{O}(e^{-\alpha x}),&x \uparrow \infty,\\
\end{cases}
\end{equation*}
uniformly in the time interval $[0,T)$.
\end{corollary}

Theorem 3.1, 3.3 and Corollary 3.4 tell us the the solutions $\rho$, $m$ and $n$ can decay as $e^{-\alpha |x|}$ as $|x| \rightarrow \infty$ for any $\alpha>0$, but $u$ and $v$ only holds for $\alpha \in (0,1)$. The reason is there is convolutions in (27), while Lemma 3.2 only holds for $\alpha \in (0,1)$. Thus whether the decay of $u$ and $v$ is holds for critical point $\alpha =1$?

\begin{theorem}
Let $(\rho_0,u_0,v_0) \in H^s \times H^s \times H^{s-1}$ with $s>\frac{5}{2}$, $T>0$ be the maximal existence time, and the corresponding solution $(\rho,u,v) \in C([0,T);H^s \times H^s \times H^{s-1})$. If the initial data satisfy
$$\|\rho_0\|_{L^\infty}+\|\rho_{0x}\|_{L^\infty}+\|u_0\|_{L^\infty}+\|u_{0x}\|_{L^\infty}+\|v_0\|_{L^\infty}+\|v_{0x}\|_{L^\infty} \le ce^{-|x|},$$
then we have
$$\|\rho(t)\|_{L^\infty}+\|\rho_x(t)\|_{L^\infty}+\|u(t)\|_{L^\infty}+\|u_x(t)\|_{L^\infty}+\|v(t)\|_{L^\infty}+\|v_x(t)\|_{L^\infty} \le ce^{-|x|}.$$
\end{theorem}

\begin{proof}
As Lemma 3.2 fails for $\alpha=1$, we introduce a new weighted function $\phi(x)=\min\{e^{|x|},N\}$, $N \in \mathbb{N}^*$. Denote $\eta =e^{|x|}$, then $\phi(x)$ is called $\eta$-moderate and $\phi^\frac{1}{2}(x)$ is $\eta^\frac{1}{2}$-moderate. With a similar argument in the proof of Theorem 3.3, we can obtain the differential inequalities of $\|\phi^\frac{1}{2}\rho\|_{L^{2k}}$, $\|\phi^\frac{1}{2}\rho_x\|_{L^{2k}}$, $\|\phi^\frac{1}{2}u|_{L^{2k}}$, $\|\phi^\frac{1}{2}u_x\|_{L^{2k}}$, $\|\phi^\frac{1}{2}v\|_{L^{2k}}$ and $\|\phi^\frac{1}{2}v_x\|_{L^{2k}}$. For example, we give a detailed estimation of $\|\phi^\frac{1}{2}u|_{L^{2k}}$. Substituting weighted function $\phi^\frac{1}{2}$ for $\phi_N$ in (47), we have
\begin{equation}
\frac{d}{dt}\|\phi^\frac{1}{2}u\|_{L^{2k}} \le C\|\phi^\frac{1}{2}u\|_{L^{2k}}+\|\phi^\frac{1}{2}(g*F)\|_{L^{2k}}+\|\phi^\frac{1}{2}(\partial_xg*uu_xv_x)\|_{L^{2k}},
\end{equation}
where $C=C(s,k,M)>0$. Note that $|\partial_xg| \le |g|=\frac{1}{2}e^{-|x|}$ a.e., and $\phi^\frac{1}{2}$ is $\eta^\frac{1}{2}$ moderate, we have
\begin{equation}
\begin{aligned}
&\|\phi^\frac{1}{2}g*F\|_{L^{2k}} \le \|\eta^\frac{1}{2}g\|_{L^1}\|\phi^\frac{1}{2}F\|_{L^{2k}}\le C\Big(\|\phi^\frac{1}{2}u\|_{L^{2k}}+\|\phi^\frac{1}{2}u_x\|_{L^{2k}}\Big),\\
&\|\phi^\frac{1}{2}\partial_xg*uu_xv_x\|_{L^{2k}} \le \|\eta^\frac{1}{2}\partial_xg\|_{L^1}\|\phi^\frac{1}{2}uu_xv_x\|_{L^{2k}} \le C\|\phi^\frac{1}{2}u\|_{L^{2k}}.
\end{aligned}
\end{equation}
Inserting (59) into (58) is given by
\begin{equation}
\frac{d}{dt}\|\phi^\frac{1}{2}u\|_{L^{2k}} \le C\Big(\|\phi^\frac{1}{2}u\|_{L^{2k}}+\|\phi^\frac{1}{2}u\|_{L^{2k}}\Big).
\end{equation}
Similar to (60), we claim that
\begin{equation*}
\begin{aligned}
&\frac{d}{dt}\Big(\|\phi^\frac{1}{2}\rho\|_{L^{2k}}+\|\phi^\frac{1}{2}\rho_x\|_{L^{2k}}+\|\phi^\frac{1}{2}u\|_{L^{2k}}+\|\phi^\frac{1}{2}u_x\|_{L^{2k}}+\|\phi^\frac{1}{2}v\|_{L^{2k}}+\|\phi^\frac{1}{2}v_x\|_{L^{2k}}\Big)\\
\le &C(\|\phi^\frac{1}{2}\rho\|_{L^{2k}}+\|\phi^\frac{1}{2}\rho_x\|_{L^{2k}}+\|\phi^\frac{1}{2}u\|_{L^{2k}}+\|\phi^\frac{1}{2}u_x\|_{L^{2k}}+\|\phi^\frac{1}{2}v\|_{L^{2k}}+\|\phi^\frac{1}{2}v_x\|_{L^{2k}}\Big).
\end{aligned}
\end{equation*}
By the Gronwall's inequality to obtain
\begin{equation*}
\begin{aligned}
&\Big(\|\phi^\frac{1}{2}\rho\|_{L^{2k}}+\|\phi^\frac{1}{2}\rho_x\|_{L^{2k}}+\|\phi^\frac{1}{2}u\|_{L^{2k}}+\|\phi^\frac{1}{2}u_x\|_{L^{2k}}+\|\phi^\frac{1}{2}v\|_{L^{2k}}+\|\phi^\frac{1}{2}v_x\|_{L^{2k}}\Big)\\
\le &e^{CT}(\|\phi^\frac{1}{2}\rho_0\|_{L^{2k}}+\|\phi^\frac{1}{2}\rho_{0x}\|_{L^{2k}}+\|\phi^\frac{1}{2}u_0\|_{L^{2k}}+\|\phi^\frac{1}{2}u_{0x}\|_{L^{2k}}+\|\phi^\frac{1}{2}v_0\|_{L^{2k}}+\|\phi^\frac{1}{2}v_{0x}\|_{L^{2k}}\Big).
\end{aligned}
\end{equation*}

Applying (58) with $k=\infty$ and $\phi$ weight,
\begin{equation}
\frac{d}{dt}\|\phi u\|_{L^\infty} \le C\|\phi u\|_{L^\infty}+\|\phi g*F\|_{L^\infty}+\|\phi\partial_xg*uu_xv_x\|_{L^\infty}.
\end{equation}
As $\phi$ is $\eta$-moderate, we have
\begin{equation}
\begin{aligned}
\|\phi g*F\|_{L^\infty} \le &\|\eta g\|_{L^\infty}\|\phi F\|_{L^1}\\
\le &C\Big(\|\phi^\frac{1}{2}u\|_{L^2}\|\phi^\frac{1}{2}u_x\|_{L^2}+\|\phi^\frac{1}{2}u_x\|_{L^2}^2+\|\phi^\frac{1}{2}u\|_{L^2}\|\phi^\frac{1}{2}v_x\|_{L^2}+\|\phi^\frac{1}{2}\rho\|_{L^2}^2\Big)\\
\le &Ce^{CT},
\end{aligned}
\end{equation}
\begin{equation}
\|\phi\partial_xg uu_xv_x\|_{L^\infty} \le \|\eta g\|_{L^\infty}\|\phi uu_xv_x\|_{L^1}\le C\|\phi^\frac{1}{2}u\|_{L^2}\|\phi^\frac{1}{2}u_x\|_{L^2}\le Ce^{CT}.
\end{equation}
Taking (62) and (63) into (61) yields
\begin{equation}
\frac{d}{dt}\|\phi u\|_{L^\infty} \le C\|\phi u\|_{L^\infty}+Ce^{CT}.
\end{equation}
We can deal with $\phi\rho$, $\phi\rho_x$, $\phi u_x$, $\phi v$ and $\phi v_x$ as (64) and end up with
\begin{equation}
\frac{d}{dt}B(t) \le CB(t)+Ce^{CT},
\end{equation}
where $B(t)=\|\phi\rho\|_{L^\infty}+\|\phi\rho_x\|_{L^\infty}+\|\phi u\|_{L^\infty}+\|\phi u_x\|_{L^\infty}+\|\phi v\|_{L^\infty}+\|\phi v_x\|_{L^\infty}$.
By virtue of Gronwall's inequality to (65) we complete the proof of Theorem 3.4.
\end{proof}

\begin{remark}
Theorem 3.1, 3.3, 3.5 and Corollary 3.4 tell us the solution $u$ and $v$ and only exponential decay for $\alpha \in (0,1]$, while for $\rho$, $m$ and $n$, the constant $\alpha>0$ is arbitrary. Moreover, if the initial data satisfy
$$\Big(\partial_x^j \rho_0,\partial_x^j u_0,\partial_x^j v_0\Big) \sim \mathcal{O}(e^{-\alpha|x|}),\ \ \ as\ \ \ |x| \rightarrow \infty,$$
it follows that
$$\Big(\partial_x^j \rho,\partial_x^j u,\partial_x^j v\Big) \sim \mathcal{O}(e^{-\alpha|x|}),\ \ \ as\ \ \ |x| \rightarrow \infty,$$
as Sobolev index $s$ large enough.
\end{remark}

\begin{remark}
If we choose weighted function
$$\psi(x)=\psi_{a,b,c,d}(x)=e^{a|x|^b}(1+|x|)^c\log(e+|x|)^d,$$
then it follows that the solution is exponential decay, power decay and logarithmic decay. More details can be found in [23].
\end{remark}

As the first equation of 3NS is in conservation form, it follows that we can weaken the decay conditions in Theorem 3.3, 3.5 and Corollary 3.4.
\begin{theorem}
Let $(\rho_0,u_0,v_0) \in H^s \times H^s \times H^{s-1}$ with $s>\frac{5}{2}$, $T>0$ be the maximal existence time, and the corresponding solution $(\rho,u,v) \in C([0,T);H^s \times H^s \times H^{s-1})$. If there exists $\alpha \in (0,1]$ and $\beta \in (\frac{\alpha}{3},\alpha)$ such that the initial data satisfy
$$\rho_0(x) \sim o(e^{-\alpha|x|}),\ \ \ \rho_{0x}(x),u_0(x),u_{0x}(x),v_0(x),v_{0x}(x) \sim \mathcal{O}(e^{-\beta|x|}), \ \ \ |x|\rightarrow \infty,$$
then
$$\rho(t,x) \sim o(e^{-\alpha|x|}), \ \ \ |x|\rightarrow \infty.$$
\end{theorem}

\begin{proof}
Without loss of generality, we only consider $x \rightarrow +\infty$. Integrating 3NS$_1$ with respect to $t$, we have
\begin{equation}
\rho(t,x)-\rho_0(x)+\int_0^t (uv\rho_x)(s,x)ds+\int_0^t ((u_xv+uv_x)\rho)(s,x)ds=0.
\end{equation}
By virtue of Theorem 3.3, we have already proved
$$\rho(t,x),\rho_x(t,x),u(t,x),u_x(t,x),v(t,x),v_x(t,x) \sim \mathcal{O}(e^{-\beta x}), \ \ \ as \ x \rightarrow +\infty,$$
uniformly in the time interval $[0,T)$. Thus,
\begin{equation*}
\begin{aligned}
\int_0^t uv\rho_x(s,x)ds, \int_0^t ((u_xv+uv_x)\rho)(s,x)ds &\sim \mathcal{O}(e^{-3\beta x})\\
&\sim o(e^{-\alpha x}).
\end{aligned}
\end{equation*}
Together with (66) and the initial condition $\rho_0(x) \sim o(e^{-\alpha x})$ as $x \rightarrow +\infty$, we complete the proof.
\end{proof}

\section{Traveling wave solutions}
In this section, we will establish that 3NS doesn't admit nontrivial symmetrical traveling wave solutions. As we know, many CH type system admits peakon (peaked traveling wave solution) with form $p(t)e^{-|x-ct|}$, where $p(t)$ is amplitude and $c$ is velocity. It is obvious the traveling wave solution in this form is symmetrical with axis $x=ct$, and it is a weak solution, because it has no derivative at point $x=ct$, and the profile remains smooth at both interval $x<ct$ and $x>ct$. Thus we give two important definitions at first.

\begin{definition}(Symmetrical solution)
The solution $z(t,x)=(\rho,u,v)$ to 3NS is $x$-symmetric if there exists a function $b(t) \in C^1(0,+\infty)$ such that
$$z(t,x)=z(t,2b(t)-x),\ \ \ \forall t \in [0,\infty),\ \ \ a.e.,$$
the function $b(t)$ is called the symmetric axis of $z(t,x)$.
\end{definition}

\begin{definition}(Weak solution)
Let $\mathcal{F}(\mathbb{R})=\{z:z=(\rho,u,v) \in C([0,\infty),H^1 \times H^2 \times H^2)\}$. The solution $z(t,x)$ is called weak, if for any test function $\phi \in C_0^\infty([0,\infty)\times \mathbb{R})$ and $z(t,x) \in \mathcal{F}(\mathbb{R})$ satisfies

\begin{equation}
\begin{aligned}
&<\rho,\phi_t>+<\rho uv,\phi_x>=0,\\
&<u,(1-\partial_x^2)\phi_t>-<4uvu_x+\rho^2u-uu_xv_{xx},\phi>+<2uu_xv_x,\phi_x>+<uvu_x,\phi_{xx}>=0,\\
&<v,(1-\partial_x^2)\phi_t>-<4uvv_x-\rho^2v-u_{xx}vv_x,\phi>+<2u_xvv_x,\phi_x>+<uvv_x,\phi_{xx}>=0,
\end{aligned}
\end{equation}
where $<,>$ denotes the distributions on $(t,x).$
\end{definition}

Before we show our main result in this section, let's derive the following lemma firstly.

\begin{lemma}
Suppose $Z(x)=(P,U,V) \in \mathcal{F}(\mathbb{R})$. If for any test function $\phi \in C_0^\infty([0,\infty);\mathbb{R})$, the following integral equations holds
\begin{equation}
\begin{aligned}
&<-cP,\phi_x>+<PUV,\phi_x>=0,\\
&<-cU,(1-\partial_x^2)\phi_x>-<4UVU_x+P^2U-UU_xV_{xx},\phi>+<2U_xV_xU,\phi_x>+<UVU_x,\phi_{xx}>=0,\\
&<-cV,(1-\partial_x^2)\phi_x>-<4UVV_x-P^2U-U_{xx}VV_x,\phi>+<2U_xV_xV,\phi_x>+<UVV_x,\phi_{xx}>=0,
\end{aligned}
\end{equation}
then for any fixed $t_0>0$, $z(t,x)=Z(x-c(t-t_0))$ is a weak solution for 3NS.
\end{lemma}

\begin{proof}
Arguing by density, it is sufficient to consider the test functions belongs to $C_0^1([0,\infty);C_0^3(\mathbb{R}))$. Without loss of generality, let $t_0=0$. Assume $\psi \in C_0^1([0,\infty);C_0^3(\mathbb{R}))$, and define $\psi_c=\psi(t,x+ct)$, it is obvious that
\begin{equation}
\partial_c(\psi_c)=(\psi_x)_c,\ \ \ \partial_t(\psi_c)=(\psi_t)_c+c(\psi_x)_c.
\end{equation}
Suppose $z(t,x)=Z(x-c(t-t_0))$. It is easy to check
\begin{equation}
\begin{cases}
&<u,(1-\partial_x^2)\psi_t>=<U,(1-\partial_x^2)(\psi_t)_c>=<U,(1-\partial_x^2)[\partial_t(\psi_c)-c(\psi_c)_x]>,\\
&<uvu_x,\psi>=<UVU_x,\psi_c>,\ \ \ <\rho^2u,\psi>=<P^2U,\psi_c>,\\
&<uu_xv_{xx},\psi>=<UU_xV_{xx},\psi_c>,\ \ \ <u_xv_xu,\psi_x>=<U_xV_xU,(\psi_c)_x>,\\
&<uvu_x,\psi_{xx}>=<UVU_x,(\psi_c)_{xx}>.
\end{cases}
\end{equation}
As $Z=Z(x)$ is independent of parameter $t$, thus we can choose $T>0$ large enough such that it doesn't belong to the support of $\psi_c$, which leads
\begin{equation}
\begin{aligned}
<U,(1-\partial_x^2)\partial_t(\psi_c)>&=\int_\mathbb{R}U(x)\int_0^\infty (1-\partial_x^2)\partial_t(\psi_c)dtdx\\
&=\int_\mathbb{R}U(x)[(1-\partial_x^2)\psi_c(T,x)-(1-\partial_x^2)\psi_c(0,x)]dx\\
&=0.
\end{aligned}
\end{equation}
Taking (70) and (71) into (67)$_2$, we have
\begin{equation*}
\begin{aligned}
&<u,(1-\partial_x^2)\phi_t>-<4uvu_x+\rho^2u-uu_xv_{xx},\phi>+<2uu_xv_x,\phi_x>+<uvu_x,\phi_{xx}>\\
&=<-cU,(1-\partial_x^2)\phi_x>-<4UVU_x+P^2U-UU_xV_{xx},\phi>+<2U_xV_xU,\phi_x>+<UVU_x,\phi_{xx}>\\
&=0,
\end{aligned}
\end{equation*}
where we have applied (68)$_2$ with $\psi(x)=\psi_c(t,x)$. Therefore $u(t,x)=U(x-c(t-t_0))$ is a weak solution of 3NS$_2$. With a similar argument, we can prove $z(t,x)=Z(x-c(t-t_0))$ is a weak solution to 3NS. This completes the proof of Lemma 4.3.
\end{proof}

Finally, we state the main result in this section.

\begin{theorem}
If $z=(\rho,u,v)$ is a unique weak traveling wave solution of 3NS, then it is not $x$-symmetrical unless $\rho\equiv 0$.
\end{theorem}
\begin{proof}
Suppose $z(t,x)$ is an $x$-symmetrical solution of 3NS. It is sufficient to choose test function $\varphi \in C_0^1([0,\infty);C_0^3(\mathbb{R}))$. Let
$$\varphi_b(t,x)=\varphi(t,2b(t)-x),\ \ \ b(t) \in C^1(\mathbb{R}).$$
It is easy to check that $\varphi_b$ has the following properties
\begin{equation}
\begin{cases}
&(\varphi_b)_b=\varphi,\ \ \ \partial_x(\varphi_b)=-(\partial_x\varphi)_b,\\
&\partial_t(\varphi_b)=(\partial_t\varphi)_b+2\dot{b}(\partial_x\varphi)_b,
\end{cases}
\end{equation}
where $\dot{b}$ denotes the derivative with respect to parameter $t$, and moreover,
\begin{equation}
\begin{cases}
&<u_bv_b\partial_xu_b,\varphi>=-<uvu_x,\varphi_b>,\ \ \ <u_b(\partial_xu_b)(\partial_x^2v_b),\varphi>=-<uu_xv_{xx},\varphi_b>,\\
&<\rho_b^2u_b,\varphi>=<\rho^2 u,\varphi_b>,\ \ \ <(\partial_xu_b)(\partial_xv_b)u_b, \varphi_x>=-<u_xv_xu,\partial_x\varphi_b>,\\
&<u_bv_b\partial_xu_b,\varphi_{xx}>=-uvu_x,\partial_x^2\varphi_b>.
\end{cases}
\end{equation}

Since $z$ is $x$-symmetrical, combining (67)$_2$, (72) and (73), we have
\begin{equation}
\begin{aligned}
&<u,(1-\partial_x^2)\varphi_t>-<4uvu_x+\rho^2u-uu_xv_{xx},\varphi>+<2uu_xv_x,\varphi_x>+<uvu_x,\varphi_{xx}>\\
=&<u,(1-\partial_x^2)(\partial_t\varphi_b+2\dot{d}\partial_x\varphi_b>-<-4uvu_x+\rho^2u+uu_xv_{xx},\varphi_b>\\
&-<2u_xv_xu,\partial_x\varphi_b>-<uvu_x,\partial_x^2\varphi_b>=0.
\end{aligned}
\end{equation}
Taking place $\varphi$ by $\varphi_b$ in (74), due to $(\varphi_b)_b=\varphi$, it follows that
\begin{equation}
<u,(1-\partial_x^2)(\partial_t\varphi+2\dot{b}\partial_x\varphi>+<4uvu_x-\rho^2u-uu_xv_{xx},\varphi>-<2u_xv_xu,\partial_x\varphi>-<uvu_x,\partial_x^2\varphi>=0.
\end{equation}
Subtracting (75) form (74), we obtain
\begin{equation}
<u,2\dot{b}\partial_x(1-\partial_x^2)\varphi>+<8uvu_x-2uu_xv_{xx},\varphi>-<4u_xv_xu,\partial_x\varphi>-<2uvu_x,\partial_x^2\varphi>=0.
\end{equation}

If we choose $\varphi_\epsilon(t,x)=\psi(x)\eta_\epsilon(t)$ in (76), where $\psi \in C_0^\infty(\mathbb{R})$ and $\eta_\epsilon \in C_0^\infty [0,\infty)$ is a mollifier that
$$\eta_\epsilon \rightarrow \delta(t-t_0),\ \ \ \epsilon \rightarrow 0.$$
This implies that
\begin{equation}
\begin{aligned}
&\int_\mathbb{R}\partial_x(1-\partial_x^2)\psi\int_0^\infty \dot{b}u\eta_\epsilon(t)dtdx+\int_\mathbb{R}\psi\int_0^\infty(4uvu_x-uu_xv_{xx})\eta_\epsilon(t)dtdx\\
-&\int_\mathbb{R}\partial_x\psi\int_0^\infty 2u_xv_xu\eta_\epsilon(t)dtdx-\int_\mathbb{R}\partial_x^2\psi\int_0^\infty uvu_x \eta_\epsilon(t)dtdx=0.
\end{aligned}
\end{equation}
Finally, letting $\epsilon \rightarrow 0$ in (77) yields
\begin{equation}
\begin{aligned}
&<-\dot{b}(t_0)u(t_0,x),\partial_x(1-\partial_x^2)\psi>-<(4uvu_x-uu_xv_{xx})(t_0,x),\psi>\\
&\ \ \ +<2u_xv_xu(t_0,x),\partial_x\psi>+<uvu_x(t_0,x),\partial_x^2\psi>=0.
\end{aligned}
\end{equation}
Set $c=\dot{b}(t_0)$, and comparing with (68)$_2$, we see if the traveling solution $u$ is $x$-symmetrical, we must constrain the distribution $<\rho^2u,\psi>\equiv0$, which means $\rho^2u\equiv 0$ by the arbitrariness of $\psi$. Similarly, one can check that $\rho^2v\equiv 0$ also holds for $(t,x)$ a.e.
\end{proof}

In fact, by virtue of $\rho^2u\equiv\rho^2v\equiv0$, if we consider the reduction $\rho\equiv0$, we have the following corollary.
\begin{corollary}
Suppose $z(t,x)$ is $x$-symmetrical. If $z=(u,v)$ is a unique weak solution of 2NS, then $z(t,x)$ is a traveling wave.
\end{corollary}

\end{document}